\documentclass[12pt]{article}
\usepackage[centertags]{amsmath}
\usepackage{amsfonts}
\usepackage{amssymb}
\usepackage{latexsym}
\usepackage{amsthm}
\usepackage{newlfont}
\usepackage{graphicx}
\usepackage{listings}
\usepackage{booktabs}
\usepackage{abstract}
\usepackage{xcolor}
\date{}

\newlength{\defbaselineskip}
\newcommand{\setlinespacing}[1]%
           {\setlength{\baselineskip}{#1 \defbaselineskip}}

\newcommand{\N}{{\mathbb{N}}}

\newcommand{\actaqed}{\hfill $\actabox$}
{\medskip\noindent \textit{Proof of #1. }}%
{\actaqed \medskip}

\def\C{{\mathcal C}}

\def \Tr{\mathcal T}

\def \V{\mathcal V}

\def \cM{\mathcal M}
\def\R{{\mathbb R}}
\def\Z{\mathbb Z}

\def \T{\mathbb T}

\def\bbC{\mathbb C}
\def \<{\langle}
\def\>{\rangle}

\def \e{\varepsilon}

\def \ro{\varrho}
\def \de{\delta}

\def\la{\lambda}
\def\La{\Lambda}

\def\bx{\mathbf x}
\def\by{\mathbf y}

\def\bk{\mathbf k}
\def\bn{\mathbf n}

\def\bw{\mathbf w}

\def\bs{\mathbf s}
\def\btt{\mathbf t}
\def\bN{\mathbf N}
\def\bW{\mathbf W}
\def\bH{\mathbf H}

\def\bF{\mathbf F}

\newtheorem{Theorem}{Theorem}[section]
\newtheorem{Lemma}{Lemma}[section]

\newtheorem{Definition}{Definition}[section]
\newtheorem{Proposition}{Proposition}[section]

\numberwithin{equation}{section}

\newcommand{\be}{\begin{equation}}
\newcommand{\ee}{\end{equation}}

\begin{document}

\title{Bounds on Kolmogorov widths and sampling recovery for classes with small mixed smoothness}

\author{V. Temlyakov\thanks{University of South Carolina, Steklov Institute of Mathematics, Lomonosov Moscow State University, and Moscow Center for Fundamental and Applied Mathematics.}\, and T. Ullrich\thanks{Faculty of Mathematics, 09107 Chemnitz, Germany. }} \maketitle

\begin{abstract}
	{Results on asymptotic characteristics of classes of functions with mixed smoothness are obtained in the paper. Our main interest is in estimating the Kolmogorov widths of classes with small mixed smoothness. 
We prove the corresponding bounds for the unit balls of the trigonometric polynomials with frequencies from a 
hyperbolic cross. We demonstrate how our results on the Kolmogorov widths imply new upper bounds 
  for the optimal sampling recovery in the $L_2$ norm of functions with small mixed smoothness. 	 }
\end{abstract}

\section{Introduction}
\label{I}

Recent results on sampling discretization of integral norms of functions from finite dimensional subspaces emphasized 
importance of good upper bounds of asymptotic characteristics of the unit balls of these subspaces in the uniform norm
(see, for instance, \cite{VT159}, \cite{DPTT}, \cite{DPSTT1}, and \cite{Kos}). The entropy numbers are used in applications in sampling discretization. Very recently it was noticed (see \cite{VT183}) that good bounds on the Kolmogorov widths of a function class can be used for estimation of errors of the optimal sample recovery of functions from this class in the $L_2$ norm. Motivated by these recent discoveries we concentrate on a study of the upper bounds of  the Kolmogorov widths of classes of functions with mixed smoothness in the uniform norm. It is well known (see, for instance, the list of 
Outstanding Open Problems on page 12 of \cite{DTU}) that approximation of functions with mixed smoothness in the uniform norm is very difficult. 

Following a classical approach we begin with a study of the asymptotic characteristics of the unit balls of subspaces of the trigonometric polynomials with frequencies in a hyperbolic cross.  
In Section \ref{A} we use an elementary approach, which   applies a standard cutoff operator. A new ingredient here is that we apply the cutoff operator to the dyadic blocks of a function, not to the function itself.  Then we apply the  finite dimensional results to smoothness classes. In this paper we only consider the case of small smoothness. The corresponding results in the case of large smoothness are known. Also, it is known that a step from analysis of classes with large smoothness to analysis of classes with small smoothness is a non-trivial step, which requires a new technique. Our analysis confirms this observation. 

We now formulate some of our results in order to demonstrate the flavor of the obtained results. We formulate the corresponding results for the Kolmogorov widths: For a compact set $\bF \subset X$ of a Banach space $X$ define
$$
d_m(\bF,X) := \inf_{\{u_i\}_{i=1}^m\subset X}
\sup_{f\in \bF}
\inf_{c_i} \left \| f - \sum_{i=1}^{m}
c_i u_i \right\|_X,\quad m=1,2,\dots
$$
and
$$
d_0(\bF,X):= \sup_{f\in \bF}\|f\|_X.
$$
Let $Q_n$, $n\in \N$, be the stepped hyperbolic cross:
$$
Q_n := \cup_{\bs:\|\bs\|_1\le n} \rho(\bs),
$$
where
$$
\rho (\bs) := \{\bk \in \Z^d : [2^{s_j-1}] \le |k_j| < 2^{s_j}, \quad j=1,\dots,d\},
$$
and let the corresponding set of the hyperbolic cross polynomials be $\Tr(Q_n)$. For a finite subset 
$Q\subset \Z^d$ we denote a subspace of the trigonometric polynomials with frequencies in $Q$ by
$$
\Tr(Q) := \left\{f: f=\sum_{\bk\in Q} c_\bk e^{i(\bk,\bx)}\right\}
$$
and denote the unit ball of $\Tr(Q)$ in the $L_p$ norm by
$$
\Tr(Q)_p := \{f\in \Tr(Q_n)\,:\, \|f\|_p \le 1\}.
$$
In Section \ref{A} we prove the following upper bound for $2\le p<\infty$
\be\label{I1}
 d_m(\Tr(Q_n)_p,L_\infty) \le C(p,d) (2^n/m)^{1/p}n^{(d-1)(1-1/p)+1/p}.
\ee
In the case $p=2$ bound (\ref{I1}) is known (see \cite{TrBe}, 11.2.5, p.489, and historical comments there). We use it in the proof of (\ref{I1}). 

Bound (\ref{I1}) and Carl's inequality (see Section \ref{D}) imply the following bound for the entropy numbers (see the definition of the entropy numbers below in Section \ref{B})
\be\label{I2}
\e_k(\Tr(Q_n)_p,L_\infty) \le C(p,d) (|Q_n|/k)^{1/p}n^{(d-1)(1-2/p)+1/p}.
\ee
Bound (\ref{I2}) is known. It was obtained in \cite{VT180} to prove the Marcinkiewicz type discretization theorems for the hyperbolic cross polynomials. Note that the known proof of (\ref{I2}) is based on deep results from functional analysis
(see Section \ref{D} for a discussion). Thus, our analysis in this paper provides an alternative proof of (\ref{I2}).

We derive the following results for the mixed smoothness classes $\bW^r_p$ (see the definition below in Section \ref{Ac}) from (\ref{I1}).
Let $d\ge 2$, $2<p\le\infty$.    Then for $1/p<r<1/2$ we have
\be\label{I3}
d_m(\bW^r_p,L_\infty) \le C(p,d,r) m^{-r} (\log m)^{(d-2)(1-r)+1}
\ee
and for $r=1/2$ we have
\be\label{I4}
d_m(\bW^{1/2}_p,L_\infty) \le C(p,d) m^{-1/2} (\log m)^{d/2} (\log \log m)^{3/2}.
\ee

The main results of the paper are presented in Sections \ref{A} and \ref{Ac}. In Section \ref{B} we discuss the mostly known results in the case $d=2$. In Section \ref{C} we show how our new results from Section \ref{Ac} can be applied for estimating optimal errors of numerical integration on classes with small mixed smoothness. For example, combining known lower bounds for numerical integration and the known Novak's inequality with results from Section \ref{Ac} we obtain the following relation for $2<p<\infty$ and $1/p<r<1/2$
$$
m^{-r}(\log m)^{(d-1)/2} \ll \kappa_m(\bW^r_p)\le 2d_m(\bW^r_p,L_\infty) \ll m^{-r} (\log m)^{(d-2)(1-r)+1},
$$
where $\kappa_m(\bW^r_p)$ is the optimal error of numerical integration with $m$ knots (see Section \ref{C} for details). Also, in Section \ref{C} we demonstrate how results of Section \ref{Ac} provide new upper bounds for the sampling recovery on classes with small mixed smoothness. In Section \ref{D} we discuss a connection of our new results with the problem of sampling discretization of integral norms
of trigonometric polynomials from $\Tr(Q_n)$. 

For the reader's convenience we write $a_m\ll b_m$ instead of $a_m\le C(p,d)b_m$ or $a_m\le C(p,d,r)b_m$, where $C(p,d)$ and $C(p,d,r)$ are positive constants. In case $a_m\ll b_m$ and $b_m\ll a_m$ we write $a_m\asymp b_m$.

\section{Kolmogorov widths and hyperbolic crosses}
\label{A}

Let  $\Delta Q_n := Q_n\setminus Q_{n-1}$.  We begin with a proof of bound (\ref{I1}). The main idea of our proof is an application of the cutoff operator 
 to the dyadic blocks of a function from $\Tr(\Delta Q_n)$. The following Nikol'skii inequality for $\Tr(Q_n)$ is known (see \cite{VTbookMA}, p.161, Theorem 4.3.16): for any $f\in \Tr(Q_n)$ we have 
\be\label{N}
\|f\|_\infty \ll  2^{n/p} n^{(d-1)(1-1/p)}\|f\|_p,\quad 1\le p<\infty.
\ee

\begin{Lemma}\label{AL1} Let $2\le p<\infty$. We have for $\bar m:= \max(m,1)$
\be\label{A2}
d_m(\Tr(\Delta Q_n)_p,L_\infty) \ll (2^n/\bar m)^{1/p}n^{(d-1)(1-1/p)+1/p},\quad m=0,1,\dots,
\ee
\be\label{A2'}
d_m(\Tr(Q_n)_p,L_\infty) \ll (2^n/\bar m)^{1/p}n^{(d-1)(1-1/p)+1/p},\quad m=0,1,\dots.
\ee
\end{Lemma}
\begin{proof} If $m=0$ then Lemma \ref{AL1} follows from (\ref{N}). Assume that $m\ge 1$. In case $p=2$ (\ref{A2'}) follows from the known case of (\ref{I1}). Clearly, (\ref{A2'}) implies (\ref{A2}). It is easy to see that also (\ref{A2}) implies (\ref{A2'}). So, we concentrate on the proof of (\ref{A2}).  We need the de la Vall{\' e}e Poussin kernels
$$
\Delta \V_\bs(\bx):= \prod_{j=1}^d (\V_{2^{s_j}}(x_j) - \V_{[2^{s_j-2}]}(x_j)),
$$
where $\V_N(x)$, $N\in \N$, is the classical univariate de la Vall{\' e}e Poussin kernels (see, for instance, \cite{VTbookMA}, p.10) and $\V_0(x)=0$. 
Let $\ast$ denote the convolution and let $\Delta V_\bs$ be the convolution operator with the kernel  $\Delta \V_\bs$. For a function $f\in L_1(\T^d)$ denote
$$
\de_\bs(f)(\bx) := \sum_{\bk\in \rho(\bs)} \hat f (\bk) e^{i(\bk,\bx)},\quad \hat f (\bk):=(2\pi)^{-d}\int_{\T^d} f(\bx)e^{-i(\bk,\bx)}d\bx.
$$
Then
$$
\de_\bs(f)\ast \Delta \V_\bs = \de_\bs(f),\qquad \|\Delta V_\bs\|_{L_p\to L_p} \le C(d),\quad 1\le p\le \infty.
$$
Let $f\in \Tr(\Delta Q_n)$. The following corollary of the Littlewood-Paley theorem is well known (see, for instance, \cite{VTbookMA}, p.513): For $p\in [2,\infty)$
\be\label{A3}
  \left(\sum_{\|\bs\|_1=n} \|\de_\bs(f)\|_p^p\right)^{1/p} \ll \|f\|_p.
\ee
For parameters $T>0$ and $p\in [1,\infty)$ define $g_T$ -- the lower cutoff and $g^T$ -- the upper cutoff of $g$:
$$
g_T(\bx)=g(\bx)\quad \text{if}\quad |g(\bx)|\le T\|g\|_p; \quad g_T(\bx)=0 \quad \text{otherwise}.
$$
Define $g^T:= g-g_T$. We note that the cutoff operator (truncation) is a nonlinear operator with the following property. For a function $g$ from a given subspace we cannot guarantee that $g_T$ belongs to the same subspace. 

Consider $\de_\bs(f)_T$ and $\de_\bs(f)^T$. Clearly,
$$
\|\de_\bs(f)_T\|_\infty \le T\|\de_\bs(f)\|_p.
$$
It is easy to check that
$$
\|\de_\bs(f)^T\|_2 \le T^{1-p/2}\|\de_\bs(f)\|_p.
$$
Define
$$
t^1_\bs := \de_\bs(f)^T\ast \Delta \V_\bs,\qquad t^2_\bs := \de_\bs(f)_T\ast \Delta \V_\bs.
$$
Then $t^1_\bs +t^2_\bs = \de_\bs(f)$ and
\be\label{A4}
\|t^1_\bs\|_2 \ll T^{1-p/2}\|\de_\bs(f)\|_p,
\ee
\be\label{A5}
\|t^2_\bs\|_\infty \ll T\|\de_\bs(f)\|_p.
\ee
Denote
$$
f^i:= \sum_{\|\bs\|_1=n} t^i_\bs,\qquad i=1,2.
$$
Then
$$
\|f^2\|_\infty \le \sum_{\|\bs\|_1=n} \|t^2_\bs\|_\infty \ll T \sum_{\|\bs\|_1=n} \|\de_\bs(f)\|_p
$$
$$
\ll T\left(\sum_{\|\bs\|_1=n} \|\de_\bs(f)\|_p^p\right)^{1/p}n^{(d-1)(1-1/p)} \ll T n^{(d-1)(1-1/p)}\|f\|_p.
$$
It is not difficult to see that
$$
\|f^1\|_2^2 \ll \sum_{\|\bs\|_1=n}\|t^1_\bs\|_2^2 \ll T^{2-p}\sum_{\|\bs\|_1=n}\|\de_\bs(f)\|_p^2 
$$
$$
\ll T^{2-p}\left(\sum_{\|\bs\|_1=n}\|\de_\bs(f)\|_p^p\right)^{2/p}n^{(d-1)(1-2/p)}\ll T^{2-p} n^{(d-1)(1-2/p)}\|f\|_p^2.
$$
Thus, for any $T$ we have
$$
d_m(\Tr(\Delta Q_n)_p,L_\infty) \ll T^{1-p/2}n^{(d-1)(1/2-1/p)}d_m(\Tr(Q_{n+d})_2,L_\infty) + Tn^{(d-1)(1-1/p)}.
$$
Using (\ref{A2'}) with $p=2$, we set $T$ such that
$$
T^{1-p/2}n^{(d-1)(1/2-1/p)}(2^n/m)^{1/2}n^{d/2} = Tn^{(d-1)(1-1/p)}
$$
and obtain 
$$
T^{p/2} = (2^n/m)^{1/2}n^{1/2}\quad \text{which implies} \quad T= (2^n/m)^{1/p}n^{1/p},
$$
and proves (\ref{A2}).

\end{proof}

Let $Q=\cup_{\bs\in E} \rho(\bs)$, where $E$ is a finite set. Denote for $ f\in \Tr(Q)$
$$
\|f\|_{\bH_p}:= \max_{\bs\in E}\|\delta_\bs(f)\|_p; \quad  \Tr(Q)_{\bH_p} := \{f\in \Tr(Q)\,:\,  \|f\|_{\bH_p}\le 1\}.
$$ 
The following Nikol'skii inequality for $\Tr(\rho(\bs))$ is known (see \cite{VTbookMA}, p.90, Theorem 3.3.2): For any $f\in \Tr(\rho(\bs))$ we have 
\be\label{Nd}
\|f\|_\infty \ll  2^{\|\bs\|_1/p} \|f\|_p.
\ee
Therefore, for any $f\in \Tr(Q_n)$ we have
\be\label{Nh}
\|f\|_\infty \ll  2^{n/p}n^{d-1} \|f\|_{\bH_p}.
\ee

\begin{Lemma}\label{AL2} Let $2\le p<\infty$. We have
\be\label{A2h}
d_m(\Tr(\Delta Q_n)_{\bH_p},L_\infty) \ll (2^n/\bar m)^{1/p}n^{d-1 +1/p},\quad m=0,1,\dots,
\ee
\be\label{A2'h}
d_m(\Tr(Q_n)_{\bH_p},L_\infty) \ll (2^n/\bar m)^{1/p}n^{d-1 +1/p},\quad m=0,1,\dots.
\ee
\end{Lemma}
\begin{proof} In the case $m=0$ Lemma \ref{AL2} follows from (\ref{Nh}). Assume $m\ge 1$.  We use notations from the proof of Lemma \ref{AL1}. By (\ref{A4}) and (\ref{A5}) we obtain
\be\label{A4h}
\|t^1_\bs\|_2 \ll T^{1-p/2}\|\de_\bs(f)\|_p \le T^{1-p/2},
\ee
\be\label{A5h}
\|t^2_\bs\|_\infty \ll T\|\de_\bs(f)\|_p \le T.
\ee

Then
$$
\|f^2\|_\infty \le \sum_{\|\bs\|_1=n} \|t^2_\bs\|_\infty \ll T \sum_{\|\bs\|_1=n} 1 \ll T n^{d-1}.
$$
 It is not difficult to see that
$$
\|f^1\|_2^2 \ll \sum_{\|\bs\|_1=n}\|t^1_\bs\|_2^2 \ll T^{2-p}\sum_{\|\bs\|_1=n} 1 \ll  T^{2-p} n^{d-1}.
$$
 Thus, for any $T$ we have
$$
d_m(\Tr(\Delta Q_n)_{\bH_p},L_\infty) \ll T^{1-p/2}n^{(d-1)/2}d_m(\Tr(Q_{n+d})_2,L_\infty) + Tn^{d-1}.
$$
Using (\ref{A2'}) with $p=2$, we set $T$ such that
$$
T^{1-p/2}n^{(d-1)/2}(2^n/m)^{1/2}n^{d/2} = Tn^{d-1}
$$
and obtain 
$$
T^{p/2} = (2^n/m)^{1/2}n^{1/2}\quad \text{which implies} \quad T= (2^n/m)^{1/p}n^{1/p},
$$
and proves (\ref{A2h}).

\end{proof}

{\bf Comment 1.} Let us rewrite the right hand sides of (\ref{A2}) and (\ref{A2'}) in the form
$$
(2^n/m)^{1/p}n^{(d-1)(1-1/p)+1/p}=\left(\frac{2^n}{mn^{d-2}}\right)^{1/p} n^{d-1}
$$
and the right hand sides of (\ref{A2h}) and (\ref{A2'h}) in the form
$$
(2^n/m)^{1/p}n^{d-1 +1/p}=\left(\frac{2^nn}{m}\right)^{1/p} n^{d-1}.
$$
Then, depending on a specific $m$, the corresponding expressions either increase or decrease with 
$1/p$. Taking into account that $\Tr(Q)_p \subset \Tr(Q)_2$ for $p\ge 2$, we conclude that for some $m$ 
we obtain a better bound by applying Lemma \ref{AL1} or Lemma \ref{AL2} with $p=2$ and for other $m$ applying the lemmas with $p$. 

\section{Kolmogorov widths of classes of functions}
\label{Ac}

Lemmas \ref{AL1} and \ref{AL2} give the bound in the form
\be\label{Tr1}
d_m(A_n,L_\infty) \ll (2^n/\bar m)^{1/p}n^{a(p)}
\ee
with appropriate $A_n$ (either $\Tr(Q_n)_p$ or $\Tr(Q_n)_{\bH_p}$)  and $a(p)=a(p,d)$. Note that $a(p)$ satisfies the inequality
\be\label{Tr2}
a(p)-a(2) \ge 1/p-1/2.
\ee

 We begin with a general result. For a number $t$ and sets $A$ and $B$ denote
 $$
 tA:= \{tf\,:\,f\in A\},\qquad A\oplus B := \{f+g\,:\, f\in A,\,g\in B\}.
 $$

\begin{Theorem}\label{AT0} Let $d\ge 2$, $2<p<\infty$, and $1/p<r<1/2$. Suppose that  
$$
\bF^r_p:= \bF^r_p(\{A_n\}_{n=1}^\infty) = \bigoplus_{n=1}^\infty 2^{-rn}A_n,\quad A_n \subset \Tr(Q_n),
$$
with compact subsets $\{A_n\}$ satisfying (\ref{Tr1}) and $a(p)$ satisfying (\ref{Tr2}). 
Then
$$
d_m(\bF^r_p,L_\infty) \ll m^{-r} (\log m)^{a(2)+(1/2-r)(1/2-1/p)^{-1}(a(p)-a(2))}.
$$
\end{Theorem}
\begin{proof} We have
$$
d_m(\bF^r_p,L_\infty) \le \sum_{n=0}^\infty 2^{-rn} d_{m_n} (A_n,L_\infty)
$$ 
for any sequence $\{m_n\}$ such that $\sum_n m_n \le m$. For a given $m\in \N$ we will bound $d_{C(d,p,r)m}$ instead of $d_m$. We will have three intervals of summation, which will be treated separately (see {\bf Comment 1} for explanation). Let $n_0$ be the largest satisfying $|Q_{n_0}| \le m$. For each $n<n_0$ set $m_n =
\dim \Tr(Q_n)$. Then the corresponding sum over $[0,n_0)$ vanishes. Let $S$ and $n_1$ be two parameters, which we will specify later. For $n\in [n_0,n_1]$ we set 
$$
m_n := [2^n 2^{(n_1-n)t} S^{-1}],\quad \text{with}\quad t\,:\, 2r<t<1.
$$
Using (\ref{Tr1}) with $p=2$ we find
$$
\sum_{n=n_0}^{n_1}2^{-rn} d_{m_n}(A_n,L_\infty) \ll \sum_{n=n_0}^{n_1} 2^{-rn}(2^n/\bar m_n)^{1/2}n^{a(2)} 
$$
\be\label{A6}
\ll S^{1/2}2^{-rn_1}n_1^{a(2)}.
\ee

For $n>n_1$ we set $m_n:= [m2^{-\kappa (n-n_1)}]$ with $\kappa$ such that $r>(1+\kappa)/p$.
Then, using (\ref{Tr1}) with $p$, we get
$$
 \sum_{n>n_1} 2^{-rn} d_{m_n}(A_n,L_\infty) \ll \sum_{n>n_1}  2^{-rn}(2^n/\bar m_n)^{1/p}n^{a(p)} 
 $$
 \be\label{A7}
 \ll 2^{-rn_1}(2^{n_1}/m)^{1/p}n_1^{a(p)} .
 \ee
 We choose $n_1$ and $S$ such that
 \be\label{A8}
 \sum_{n=n_0}^{n_1} m_n \asymp 2^{n_1} /S \asymp m
 \ee
 and
 \be\label{A9}
 S^{1/2}2^{-rn_1}n_1^{a(2)} \asymp 2^{-rn_1}(2^{n_1}/m)^{1/p}n_1^{a(p)}.
 \ee
 That is $S\asymp n_1^{(a(p)-a(2))(1/2-1/p)^{-1}}$. Then we obtain
 $$
 d_m(\bF^r_p,L_\infty) \ll S^{1/2-r}m^{-r} n_1^{a(2)} \asymp m^{-r} (\log m)^{a(2)+(1/2-r)(1/2-1/p)^{-1}(a(p)-a(2)) }.
 $$
Theorem \ref{AT0} is proved.

\end{proof}

We now proceed to applications of Theorem \ref{AT0} to classes of functions of mixed smoothness.
We define the class $\bW^r_{p}$ in the following way. For $r > 0$   the functions
$$
F_r(x):=1+2\sum_{k=1}^{\infty}k^{-r}\cos(kx - r\pi/2)
$$
are called {\it Bernoulli kernels}.
 Let  
$$
F_r (\bx) := \prod_{j=1}^d F_r (x_j)
$$
be the multivariate analog of the Bernoulli kernel. 
We denote by $\bW_{p}^r$ the class of functions
$f(\bx)$ representable in the form
$$
f(\bx) =\varphi(\bx)\ast F_r (\bx) :=
(2\pi)^{-d}\int_{\T^d}\varphi(\by)F_r (\bx-\by)d\by,
$$
where $\varphi\in L_p$ and $\|\varphi\|_p\le 1$.  Note that in the case of integer $r$ the class $\bW^r_{p}$   is equivalent to the class defined by restrictions on mixed derivatives (for more details see \cite{DTU}, Ch.3). 

It is well known (see, for instance, \cite{VTbookMA}, p.174, Theorem 4.4.9) that in the case $1<p<\infty$ the class $\bW^r_p$ is embedded into the class $\bF^r_p(\{\Tr(Q_n)_p\}_{n=1}^\infty)$.
Applying Theorem \ref{AT0} in the case of $\bW^r_p$ and using Lemma \ref{AL1}, which gives (\ref{Tr1}) with 
$a(p)= (d-1)(1-1/p)+1/p$, we obtain the following bound for the $ d_m(\bW^r_p,L_\infty)$.

\begin{Theorem}\label{AT1} Let $d\ge 2$, $2<p\le\infty$, and $1/p<r<1/2$. Then
$$
d_m(\bW^r_p,L_\infty) \ll m^{-r} (\log m)^{(d-2)(1-r)+1}.
$$
\end{Theorem}

 We now turn our discussion to the classes $\bH^r_p$.
Let  $\btt =(t_1,\dots,t_d )$ and $\Delta_{\btt}^l f(\bx)$
be the mixed $l$-th difference with
step $t_j$ in the variable $x_j$, that is
$$
\Delta_{\btt}^l f(\bx) :=\Delta_{t_d,d}^l\dots\Delta_{t_1,1}^l
f(x_1,\dots ,x_d ) .
$$
Let $e$ be a subset of natural numbers in $[1,d ]$. We denote
$$
\Delta_{\btt}^l (e) =\prod_{j\in e}\Delta_{t_j,j}^l,\qquad
\Delta_{\btt}^l (\varnothing) = I .
$$
We define the class $\bH_{p,l}^r B$, $l > r$, as the set of
$f\in L_p$ such that for any $e$
\be\label{B7a}
\bigl\|\Delta_{\btt}^l(e)f(\bx)\bigr\|_p\le B
\prod_{j\in e} |t_j |^r .
\ee
In the case $B=1$ we omit it. It is known (see, for instance, \cite{VTbookMA}, p.137) that the classes $\bH^r_{p,l}$ with different $l$ are equivalent. So, for convenience we fix one $l= [r]+1$ and omit $l$ from the notation. 

It is well known (see, for instance, \cite{VTbookMA}, p.171, Theorem 4.4.6) that in the case $1<p<\infty$ the class $\bH^r_p$ is embedded into the class $\bF^r_p(\{\Tr(Q_n)_{\bH_p}\}_{n=1}^\infty)$.
Applying Theorem \ref{AT0} in the case of $\bH^r_p$ and using Lemma \ref{AL2}, which gives (\ref{Tr1}) with 
$a(p)= d-1+1/p$, we obtain the following bound for the $ d_m(\bH^r_p,L_\infty)$.

\begin{Theorem}\label{AT1h} Let $d\ge 2$, $2<p\le\infty$, and $1/p<r<1/2$. Then
$$
d_m(\bH^r_p,L_\infty) \ll m^{-r} (\log m)^{d-1+r}.
$$
\end{Theorem}

We now discuss the case $r=1/2$. As above we begin with a conditional theorem.

\begin{Theorem}\label{ATcond} Let $d\ge 2$, $2<p<\infty$. Suppose that 
$$
\bF^r_p = \bigoplus_{n=1}^\infty 2^{-n/2}A_n,\quad A_n \subset \Tr(Q_n),
$$
with $\{A_n\}$ satisfying (\ref{Tr1}) and $a(p)$ satisfying (\ref{Tr2}). 
Then
$$
d_m(\bF^r_p,L_\infty) \ll m^{-r} (\log m)^{a(2)} (\log \log m)^{3/2}.
$$
\end{Theorem}
\begin{proof} This proof repeats the argument from the proof of Theorem \ref{AT0}. For $n\in [n_0,n_1]$ we set 
$$
m_n := [2^{n_1}S^{-1}]. 
$$
Using (\ref{Tr1}) with $p=2$ we find
$$
\sum_{n=n_0}^{n_1}2^{-n/2} d_{m_n}(\Tr(\Delta Q_n)_p,L_\infty) \ll \sum_{n=n_0}^{n_1} 2^{-n/2}(2^n/\bar m_n)^{1/2}n^{a(2)} 
$$
\be\label{A10}
\ll S^{1/2}2^{-n_1/2}n_1^{a(2)}(n_1-n_0).
\ee
The third sum is estimated as in (\ref{A7}) with $r=1/2$. We choose $n_1$ and $S$ such that
 \be\label{A11}
 \sum_{n=n_0}^{n_1} m_n \asymp 2^{n_1}S^{-1}(n_1-n_0) \asymp m
 \ee
 and
 \be\label{A12}
 S^{1/2}2^{-n_1/2}n_1^{a(2)}(n_1-n_0) \asymp 2^{-n_1/2}(2^{n_1}/m)^{1/p}n_1^{a(p)}.
 \ee
 This gives the required bound. 

\end{proof}

Applying Theorem \ref{ATcond} in the case of $\bW^r_p$ and using Lemma \ref{AL1}, which gives (\ref{Tr1}) with 
$a(p)= (d-1)(1-1/p)+1/p)$, we obtain the following bound for the $ d_m(\bW^{1/2}_p,L_\infty)$.

\begin{Theorem}\label{AT2} Let $d\ge 2$, $2<p\le\infty$, and $r=1/2$. Then
$$
d_m(\bW^{1/2}_p,L_\infty) \ll m^{-1/2} (\log m)^{d/2} (\log \log m)^{3/2}.
$$
\end{Theorem}

Applying Theorem \ref{AT0} in the case of $\bH^r_p$ and using Lemma \ref{AL2}, which gives (\ref{Tr1}) with 
$a(p)= d-1+1/p)$, we obtain the following bound for the $ d_m(\bH^{1/2}_p,L_\infty)$.

\begin{Theorem}\label{AT2h} Let $d\ge 2$, $2<p\le\infty$, and $r=1/2$. Then
$$
d_m(\bH^{1/2}_p,L_\infty) \ll m^{-1/2} (\log m)^{d-1/2} (\log \log m)^{3/2}.
$$
\end{Theorem}

{\bf Comment 2.} For the possible future applications of the above approach we list the properties, which were used above.
We begin with Lemma \ref{AL1}. We derived (\ref{A2}) from (\ref{I1}) with $p=2$ formulated for $Q_{n+d}$ with the help of the following two simple properties.

{\bf P1.} For any positive $t$ and any set $A\subset L_\infty$ we have $d_m(tA,L_\infty) = td_m(A,L_\infty)$.

{\bf P2.} For any set $A\subset L_\infty$ and any positive $t$ we have
$$
d_m(A\oplus tB(L_\infty),L_\infty) \le C(d_m(A,L_\infty)+t),
$$
where $B(L_\infty)$ is the unit ball of the $L_\infty$. 

We now proceed to Theorems \ref{AT1} and \ref{AT2}. Proofs of these theorems are based on the bound (\ref{A2}) and on the following properties 

{\bf P3.} For any sequence $\{m_n\}$, $m_n\in \N$, such that $\sum_{n=1}^\infty m_n \le m$ and any sequence of sets $\{A_n\}$ we have
$$
d_m\left(\bigoplus_{n=1}^\infty A_n,L_\infty\right) \le \sum_{n=1}^\infty d_{m_n}(A_n,L_\infty).
$$

{\bf P4.} For all $n$ such that $\dim \Tr(Q_n) \le m$ we have $d_m(\Tr(Q_n)_p,L_\infty)=0$. 

\section{The case $d=2$}
\label{B}

Results of this section should be considered known. Some of them are explicitly written. In such a case we give a reference. Some results are the folklore results and others are simple corollaries of the above results. 
Consider $d$-dimensional parallelepipeds
$$
\Pi(\mathbf N,d) :=\bigl \{\mathbf a\in \Z^d  : |a_j|\le N_j,\
j = 1,\dots,d \bigr\} ,
$$
where $N_j$ are nonnegative integers and the corresponding subspaces of the trigonometric polynomials
$$
\Tr(\bN,d) := \Tr(\Pi(\bN,d)).
$$
Then $\dim \Tr(\bN,d) = \vartheta(\mathbf N) := \prod_{j=1}^d (2N_j  + 1)$.
The following finite dimensional result is well known: For any natural numbers $n,  m$,  $m < n$ we have
\be\label{B1}
d_m(B_p^n,  \ell_{\infty}^n)\le C m^{-1/p} \bigl(\ln (en/m) \bigr)^{1/p},\quad 2\le p <\infty.
\ee
The reader can find a simple proof of (\ref{B1}) in the case $p=2$ and a historical discussion in \cite{VTbookMA}
(see Theorem 2.1.11 there). Also, we refer the reader for historical comments to \cite{DTU}, p.55. We refer the reader for a further discussion to \cite{FPRU}. The case $2<p<\infty$ can be easily derived from the case $p=2$ using the cutoff operator.   Bound (\ref{B1}) and the Marcinkiewicz discretization theorem (\cite{VTbookMA}, p.102, Theorem 3.3.15) imply for $m\ge 1$
\be\label{B2}
d_m \bigl(\Tr(\mathbf N,d)_p, L_{\infty}\bigr)\ll\bigl(\vartheta(\mathbf N)/
m\bigr)^{1/p}\ln \bigl(e\vartheta(\mathbf N)/m) \bigr)^{1/p}.
\ee
The following Nikol'skii inequality for $\Tr(\bN,d)$ is known (see \cite{VTbookMA}, p.90, Theorem 3.3.2): For any $f\in \Tr(\N,d)$ we have 
\be\label{B2'}
\|f\|_\infty \ll  \vartheta(\mathbf N)^{1/p} \|f\|_p.
\ee
Let $d=2$.  Then we have for $2\le p<\infty$
$$
d_m \bigl(\Tr\bigl(\Delta Q_n)_{\bH_p}, L_{\infty}\bigr)\le
\sum_{\|\bs\|_1=n}d_{[m/n]}
\Tr\bigl(\rho(\bs)\bigr)_p,L_{\infty}\bigr).
$$ 
Using (\ref{B2}) and (\ref{B2'}) we obtain from here
\be\label{B3}
d_m \bigl(\Tr\bigl(\Delta Q_n)_{\bH_p}, L_{\infty}\bigr) \ll \left(\frac{|\Delta Q_n|}{m}\right)^{1/p}\left(\log\frac{|\Delta Q_n|}{m}\right)^{1/p}n .
\ee
We point out that Lemma \ref{AL2} gives a little better bound than (\ref{B3}). However, this does not affect the bound for the function class $\bH^r_p$.
\begin{Proposition}\label{BP1} Let $d=2$, $2\le p\le\infty$. Then, for $r>1/p$ we have
$$
d_m(\bH^r_p,L_\infty) \ll m^{-r}(\log m)^{r+1}.
$$
\end{Proposition}
\begin{proof} We have
$$
d_m(\bH^r_p,L_\infty) \ll \sum_{n=0}^\infty 2^{-rn} d_{m_n} (\Tr(\Delta Q_n)_{\bH_p},L_\infty)
$$ 
for any sequence $\{m_n\}$ such that $\sum_n m_n \le m$. For a given $m\in \N$ we will bound $d_{C(d,p,r)m}$ instead of $d_m$. We will have two intervals of summation, which will be treated separately. Let $n_0$ be the largest satisfying $|Q_{n_0}| \le m$. For each $n<n_0$ set $m_n =
\dim \Tr(\Delta Q_n)$. Then the corresponding sum over $[0,n_0)$ vanishes. 
For $n>n_0$ we set $m_n:= [|\Delta Q_n|2^{-\kappa (n-n_0)}]$ with $\kappa >1$ such that $r>\kappa/p$.
Then, using (\ref{B3}), we get
$$
 \sum_{n>n_0} 2^{-rn} d_{m_n}(\Tr(\Delta Q_n)_{\bH_p},L_\infty) \ll  2^{-rn_0}n_0 , 
 $$
 which proves Proposition \ref{BP1}. 
\end{proof}

Let $X$ be a Banach space. For a compact set $W \subset X$ we define the entropy numbers $\e_k(W,X)$:
$$
\e_k(W,X) :=   \inf \{\e : \exists y^1,\dots ,y^{2^k} \in X : W \subseteq \cup_{j=1}
^{2^k} B_X(y^j,\e)\}
$$
where $B_X(y,\e):=\{x\in X: \, \|x-y\|_X\le \e\}$.
 
With a help of Carl's inequality (\cite{C}, see also Section \ref{D}) we obtain from Proposition \ref{BP1} the following upper bounds for the entropy numbers
 \be\label{B6}
\e_m(\bH^r_p,L_\infty) \ll m^{-r}(\log m)^{r+1}.
\ee

Theorem 7.8.4 from \cite{VTbookMA} (p.374) states: For $d=2$, $1\le p\le \infty$, $r>1/p$ we have
\be\label{B7}
\e_m(\bH^r_p,L_\infty) \asymp m^{-r}(\log m)^{r+1}.
\ee

The following lemma is well known (see, for instance, \cite{VTbookMA}, Lemma 5.3.14, p.229).
\begin{Lemma}\label{BL1} Let $A$ be centrally symmetric compact in a separable Banach space $X$ and for two real numbers $r>0$ and $a \in \R$ we have
$$
d_m(A,X) \ll m^{-r}(\log m)^a
$$
and 
$$
\e_m(A,X) \gg  m^{-r}(\log m)^a.
$$
Then the following relations
$$
d_m(A,X) \asymp \e_m(A,X) \asymp  m^{-r}(\log m)^a
$$ 
hold.
\end{Lemma}

Lemma \ref{BL1} and Proposition \ref{BP1} imply the following asymptotic behavior of the Kolmogorov widths.

\begin{Theorem}\label{BT2} Let $d=2$, $2\le p \le \infty$, $r>1/p$. Then, we have
\be\label{B12}
d_m(\bH^r_p,L_\infty) \asymp m^{-r}(\log m)^{r+1}.
\ee
\end{Theorem}
Theorem \ref{BT2} provides the right order of the Kolmogorov widths $d_m(\bH^r_p,L_\infty)$
for all $2\le p \le \infty$ and $r>1/p$ in the two-dimensional case $d=2$. We do not know the right orders 
of $d_m(\bH^r_p,L_\infty)$ and $d_m(\bW^r_p,L_\infty)$ in the case $d\ge 3$. We also do not know 
the right orders of $d_m(\bW^r_p,L_\infty)$ in the case $d=2$, $2\le p\le \infty$ for small smoothness $r\le 1/2$. Note that in the case of large smoothness the right order is known. Theorem 5.3.18 on page 231 of \cite{VTbookMA} states: In the case $d=2$ we have, for $2\le p\le \infty$ and $r>1/2$,
$$
d_m(\bW^r_p,L_\infty) \asymp m^{-r}(\log m)^{r+1/2}.
$$

\section{Some lower bounds and applications}
\label{C}

We formulate a known result, which relates optimal error of numerical integration of a class with its Kolmogorov width. For a compact subset $\bF\subset \C(\Omega)$ define the best error of numerical integration with $m$ knots as follows
$$
\kappa_m(\bF) := \inf_{\xi^1,\dots,\xi^m;\lambda_1,\dots,\lambda_m} \sup_{f\in\bF}\left|\int_\Omega fd\mu - \sum_{j=1}^m \lambda_j f(\xi^j)\right|.
$$
The following inequality was proved in \cite{No} (see also \cite{NoLN})
\be\label{C1}
\kappa_m(\bF) \le 2d_m(\bF,L_\infty).
\ee
We use inequality (\ref{C1}) for obtaining some lower bounds for the Kolmogorov widths from the known lower bounds for numerical integration. 

In the case of the $\bW$ classes the following result is known (see, \cite{VT43} and \cite{VTbookMA}, p.264, Theorem 6.4.3): For $r>1/p$ we have
\be\label{C2}
\kappa_m(\bW^r_p)\gg m^{-r}(\log m)^{(d-1)/2},\quad 1\le p<\infty.
\ee
Combining (\ref{C2}) and (\ref{C1}) with Theorem \ref{AT1} we obtain the following relation for $2<p<\infty$ and $1/p<r<1/2$
$$
m^{-r}(\log m)^{(d-1)/2} \ll \kappa_m(\bW^r_p)\le 2d_m(\bW^r_p,L_\infty) \ll m^{-r} (\log m)^{(d-2)(1-r)+1}.
$$
This shows that the power decay of both the $\kappa_m$ and the $d_m$ is of order $m^{-r}$ and the exponents of the logarithmic factors 
differ by $(d-2)(1/2-r)+1/2$, which grows with $d$. 

Combining (\ref{C2}) and (\ref{C1}) with Theorem \ref{AT2} we obtain the following relation for $2<p<\infty$ and $r=1/2$
$$
m^{-r}(\log m)^{(d-1)/2} \ll \kappa_m(\bW^{1/2}_p) \le 2d_m(\bW^{1/2}_p,L_\infty)
$$
$$
 \ll m^{-1/2} (\log m)^{d/2}(\log \log m)^{3/2}.
$$
This shows that the power decay is of order $m^{-r}$ and the exponents of the logarithmic factors 
differ by $1/2$, which does not grow with $d$. 

We note that there are known bounds for the $\kappa_m(\bW^r_p)$, which are better than the above bounds. Namely, in the case $2<p\le \infty$ and $1/p<r<1/2$ we have (see \cite{VTbookMA}, p.276, Theorem 6.5.5 for $d=2$ and \cite{DTU}, p.138, Theorem 8.5.6 for all $d$)
$$
\kappa_m(\bW^r_p) \ll m^{-r} (\log m)^{(1-r)(d-1)}
$$
and in the case $r=1/2$ (see \cite{VTbookMA}, p.282, Theorem 6.5.9 for $d=2$ and \cite{DTU}, p.138, Theorem 8.5.6 for all $d$)
$$
\kappa_m(\bW^{1/2}_p) \ll m^{-1/2} (\log m)^{(d-1)/2}(\log\log m)^{1/2}.
$$

In the case of the $\bH$ classes the following result is known (see, for instance, \cite{DTU}, p.134, Theorem 8.5.1): For $1\le p\le \infty$ and $r>1/p$ we have
\be\label{C2h}
\kappa_m(\bH^r_p)\asymp m^{-r}(\log m)^{d-1}.
\ee
Combining (\ref{C2h}) and (\ref{C1}) with Theorem \ref{AT1h} we obtain the following relation for $2<p\le\infty$ and $1/p<r<1/2$
$$
m^{-r}(\log m)^{d-1} \ll \kappa_m(\bH^r_p)\le 2d_m(\bH^r_p,L_\infty) \ll m^{-r} (\log m)^{d-1+r}.
$$
This shows that the power decay is of order $m^{-r}$ and the exponents of the logarithmic factors 
differ by $r$, which does not grow with $d$. 

Combining (\ref{C2h}) and (\ref{C1}) with Theorem \ref{AT2h} we obtain the following relation for $2<p\le\infty$ and $r=1/2$
$$
m^{-1/2}(\log m)^{d-1} \ll \kappa_m(\bH^{1/2}_p)\le 2d_m(\bH^{1/2}_p,L_\infty)
$$
$$
 \ll m^{-1/2} (\log m)^{d-1/2}(\log \log m)^{3/2}.
$$
This shows that the power decay is of order $m^{-1/2}$ and the exponents of the logarithmic factors 
differ by $1/2$, which does not grow with $d$. 


We now discuss application of our results to the problem of optimal sampling recovery. Recall the setting 
 of the optimal recovery. Let $\Omega$ be a compact subset of $\R^d$ with a probability measure $\mu$ on it. For a fixed $m$ and a set of points  $\xi:=\{\xi^j\}_{j=1}^m\subset \Omega$, let $\Phi_\xi $ be a linear operator from $\bbC^m$ into $L_q(\Omega,\mu)$. With a little abuse of notation set $L_\infty(\Omega) =\C(\Omega)$ to be the space of functions continuous on $\Omega$. 
Denote for a class $\bF$ (usually, centrally symmetric and compact subset of $L_q(\Omega,\mu)$)
$$
\varrho_m(\bF,L_q) := \inf_{\text{linear}\, \Phi_\xi; \,\xi} \sup_{f\in \bF} \|f-\Phi_\xi(f(\xi^1),\dots,f(\xi^m))\|_q.
$$

The following result was recently obtained in \cite{VT183}. 
\begin{Theorem}\label{BT1} Let $\bF$ be a compact subset of $\C(\Omega)$. There exist two positive absolute constants $b$ and $B$ such that
$$
\ro_{bn}(\bF,L_2) \le Bd_n(\bF,L_\infty).
$$
\end{Theorem}

This theorem combined with the upper bounds for the Kolmogorov widths obtained in Section \ref{Ac}
gives the following bounds for the sampling recovery: Let $2<p\le\infty$ and $1/p<r<1/2$ then
\be\label{C3}
 \ro_m(\bW^r_p,L_2) \ll m^{-r} (\log m)^{(d-2)(1-r)+1}
\ee
and
\be\label{C3h}
\ro_m(\bH^r_p,L_2) \ll m^{-r} (\log m)^{d-1+r}.
\ee
In the case $r=1/2$ we obtain
\be\label{C4}
\ro_m(\bW^{1/2}_p,L_2) \ll m^{-1/2} (\log m)^{d/2}(\log \log m)^{3/2}
\ee
and
\be\label{C4h}
\ro_m(\bH^{1/2}_p,L_2) \ll m^{-1/2} (\log m)^{d-1/2}(\log \log m)^{3/2}.
\ee

Theorem \ref{BT1} was proved with a help of a classical type of algorithm -- weighted least squares. Let $X_N$ be an $N$-dimensional subspace of the space of continuous functions $\C(\Omega)$ and let $\bw:=(w_1,\dots,w_m)\in \R^m$ be a positive weight, i.e. $w_i>0$, $i=1,\dots,m$. Consider the following classical weighted least squares recovery operator (algorithm) (see, for instance, \cite{CM})
$$
 \ell 2\bw(\xi,X_N)(f):=\text{arg}\min_{u\in X_N} \|S(f-u,\xi)\|_{2,\bw},\quad \xi=\{\xi^j\}_{j=1}^m\subset \Omega,
$$
where
$$
\|S(g,\xi)\|_{2,\bw}:= \left(\sum_{\nu=1}^m w_\nu |g(\xi^\nu)|^2\right)^{1/2} .
$$
However, the proof of the upper bounds in (\ref{C3}) -- (\ref{C4h}) is not constructive. We are not aware of a constructive proof of the upper bounds of the Kolmogorov widths for the classes $\bW^r_p$ and $\bH^r_p$ in Theorems \ref{AT1}, \ref{AT1h}, \ref{AT2}, and \ref{AT2h}. More specifically, we do not know a good subspace, which provides approximation close to the Kolmogorov width. We note that for a given subspace $X_N$ we have an explicit way of calculating the weights of the weighted least squares algorithm. It is based on the Christoffel function of $X_N$ (see \cite{DPSTT2}, proof of Theorem 6.3, \cite{LT}, Remark 3.1, and \cite{NSU}).

There are constructive methods for the sampling recovery based on sparse grids (Smolyak point sets $SG(n)$) (see \cite{DTU}, Chapter 5 and \cite{VTbookMA}, Section 6.9). For instance, these methods give the following upper bounds for the sampling recovery for all $r>1/p$, $2\le p\le \infty$ (see  \cite{VTbookMA}, p.307, Theorem 6.9.2).
\be\label{C5}
\ro_m(\bH^{r}_p,L_2) \ll m^{-r} (\log m)^{(d-1)(1+r)} .
\ee
Clearly, for $d>2$ bound (\ref{C5}) is not as good as bounds (\ref{C3h}) and (\ref{C4h}). However, it is known that 
bound (\ref{C5}) cannot be improved by methods based on sparse grids or more generally based on $(n,l)$-nets (see below). We discuss this interesting phenomenon in detail. First of all, let us make a simple well known observation on a relation between sampling recovery and numerical integration. Associate with the recovery operator 
$$
\Psi(f,\xi):=\sum_{j=1}^m f(\xi^j)\psi_j(\bx)
$$
 the cubature formula 
 $$
 \La_m(f,\xi) := \sum_{j=1}^m \la_j f(\xi^j),\quad \la_j:= \int_\Omega \psi_j(\bx)d\mu
$$
with knots $\xi =\{\xi^j\}_{j=1}^m$ and weights $\La=\{\la_j\}_{j=1}^m$. Then for a normalized (probabilistic) measure $\mu$ we have
$$
\left|\La_m(f,\xi)-\int_\Omega f(\bx)d\mu\right| = \left|\int_\Omega (\Psi(f,\xi)-f)d\mu \right|
$$
\be\label{ni<sr}
\le \|\Psi(f,\xi)-f\|_1 \le \|\Psi(f,\xi)-f\|_q,\qquad q\ge 1.
\ee
We now present some known results on the lower bounds for the numerical integration  with respect to a special class of knots. Let $\bs = (s_1,\dots,s_d)$, $s_j\in \N_0$, $j=1,\dots,d$. We associate with $\bs$ a web
$W(\bs)$ as follows: denote 
$$
w(\bs,\bx) := \prod_{j=1}^d \sin (2^{s_j}x_j)
$$
and define
$$
W(\bs) := \{\bx: w(\bs,\bx)=0\}.
$$
\begin{Definition}\label{CD1} We say that a set of knots $\xi:=\{\xi^i\}_{i=1}^m$ is an $(n,l)$-net if $|\xi\setminus W(\bs)| \le 2^l$ for all $\bs$ such that $\|\bs\|_1=n$.
\end{Definition}

It is clear that the bigger the parameter $l$ the larger the set of $(n,l)$-nets.

\begin{Definition}\label{CD2} For $n\in\N$ we define the sparse grid $SG(n)$ as follows
$$
SG(n) := \{\xi(\bn,\bk) = (2\pi k_12^{-n_1},\dots,2\pi k_d 2^{-n_d}),
$$
$$
 0\le k_j<2^{n_j}, j=1,\dots,d,\quad \|\bn\|_1=n\}.
$$
\end{Definition}
Then it is easy to check that $SG(n)\subset W(\bs)$ with any $\bs$ such that $\|\bs\|_1=n$. This means that $SG(n)$ is an $(n,l)$-net for any $l$. Also a union of the set $SG(n)$ with any set consisting of $2^l$ points is  an $(n,l)$-net. 

For convenience, let us denote by $\ro_m^{n}$ the optimal error of sampling recovery algorithms, which use the $(n,n-1)$-nets of cardinality $m$ and by $\ro_m^{sg}$ the optimal error of sampling recovery algorithms, which use the sparse grids $SG(n)$ with cardinality $m=|SG(n)|$. In both of these cases we can take $m\asymp 2^n n^{d-1}$.

Let us begin our discussion with the $\bH$ classes. It was demonstrated in \cite{VT150} that the example constructed in \cite{VT43} for proving the lower bound (\ref{C2}) shows that the upper bound (\ref{C5}) cannot be improved if we use a special class of point sets -- $(n,l)$-nets:
For any cubature formula $\Lambda_m(\cdot,\xi)$ with respect to a $(n,n-1)$-net $\xi$ we have
$$
\sup_{f\in \bH_{p}^r}\left|\Lambda_m(f,\xi)-\int_{\T^d}f(\bx)d\bx\right| \gg 2^{-rn}n^{d-1},\quad 1\le p\le\infty.
$$
This lower bound and inequality (\ref{ni<sr}) with $q=2$ imply
\be\label{C5l}
\ro_m^{sg}(\bH^{r}_p,L_2)\ge \ro_m^{n}(\bH^{r}_p,L_2) \gg m^{-r} (\log m)^{(d-1)(1+r)} .
\ee
This inequality shows that if we use the sparse grids set $SG(n)$ of points for recovery then we cannot get a better error than in  (\ref{C5l}). Moreover, this inequality shows that even if we use the sparse grids set $SG(n)$ of points combined with any set of $2^{n-1}$ poins for recovery then we still cannot get a better error than in  (\ref{C5l}). 
The fact that (\ref{C5}) cannot be improved for the sparse grids was proved in \cite{DU}. For further discussion we refer the reader to \cite{DTU}, Ch. 5. Comparing inequalities (\ref{C5l}) and inequalities 
(\ref{C3h}) and (\ref{C4h}), we conclude that in the range of parameters $1/p<r\le 1/2$, $2< p<\infty$, 
$d>2$, there exists a weighted least squares algorithm, which provides better (albeit, nonconstructive) upper bounds for sampling recovery than algorithms based on sparse grids or even based on a wider class of point sets -- the $(n,n-1)$-nets.

The sampling recovery of the $\bW$ classes turns out to be a more difficult problem than the sampling recovery of the $\bH$ classes.  The following result is from \cite{DU} (sparse grids) and from  \cite{VT150} ($(n,l)$-nets):
For any cubature formula $\Lambda_m(\cdot,\xi)$ with respect to a $(n,n-1)$-net $\xi$, in particular with respect to the $SG(n)$ set, we have
$$
\sup_{f\in \bW_{p}^r}\left|\Lambda_m(f,\xi)-\int_{\T^d}f(\bx)d\bx\right| \gg 2^{-rn}n^{(d-1)/2},\quad 1\le p<\infty.
$$
This lower bound and inequality (\ref{ni<sr}) imply
\be\label{CWl}
\ro_m^{sg}(\bW^{r}_p,L_2)\ge \ro_m^{n}(\bW^{r}_p,L_2) \gg m^{-r} (\log m)^{(d-1)(1/2+r)} .
\ee
Comparing inequalities (\ref{CWl}) and inequalities 
(\ref{C3}) and (\ref{C4}), we conclude that in the range of parameters $1/4<r\le 1/2$, $2< p<\infty$, for large enough $d$
there exists a weighted least squares algorithm, which provides better (albeit, nonconstructive) upper bounds for sampling recovery than algorithms based on sparse grids.

The reader can find recent results on optimal sampling recovery in the papers \cite{KU}, \cite{NSU}, \cite{VT183}, and \cite{KU2}.

Discussions in Sections \ref{B} and \ref{C} show that we have made some progress in obtaining the upper bounds for the Kolmogorov widths of classes with small mixed smoothness but the right orders of 
them are still not established. We formulate two open problems in this regard.

{\bf Open problem 1.} Find the right orders of decay of $d_m(\bW^r_p,L_\infty)$ in the case $d\ge 2$,
$2< p\le\infty$, $1/p<r\le 1/2$.

 {\bf Open problem 2.} Find the right orders of decay of $d_m(\bH^r_p,L_\infty)$ in the case $d\ge 3$,
$2< p\le\infty$, $1/p<r\le 1/2$.

Note that Theorem \ref{BT2} gives the right orders of decay of $d_m(\bH^r_p,L_\infty)$ in the case $d=2$, and $r>1/p$, which means that for $\bH$ classes the problem is solved in dimension $d=2$.

\section{Discussion}
\label{D}

There are several general results, which give 
lower estimates
on the Kolmogorov widths $d_n(F,X)$ in terms of the entropy numbers
$\e_k(F,X)$.  Carl's (see \cite{C} and \cite{VTbook}, p.169, Theorem 3.23)
inequality states: For any $r>0$ we have
\begin{equation}\label{D1}
\max_{1\le k \le n} k^r \e_k(F,X) \le C(r) \max _{1\le m \le n} m^r d_{m-1}(F,X).
\end{equation}
Inequality (\ref{D1}) and Lemma \ref{AL1} imply
\begin{Lemma}\label{DL1} Let $2\le p<\infty$. We have
\be\label{D2}
\e_k(\Tr(Q_n)_p,L_\infty) \ll (|Q_n|/k)^{1/p}n^{(d-1)(1-2/p)+1/p}.
\ee
\end{Lemma}
Bound (\ref{D2}) is known. It was obtained in \cite{VT180} to prove the Marcinkiewicz type discretization theorems for the hyperbolic cross polynomials. For the reader's convenience we describe these results here. 

{\bf The Marcinkiewicz discretization problem.} Let $\Omega$ be a compact subset of $\R^d$ with the probability measure $\mu$. We say that a linear subspace $X_N$ (index $N$ here, usually, stands for the dimension of $X_N$) of $L_q(\Omega)$, $1\le q < \infty$, admits the Marcinkiewicz type discretization theorem with parameters $m\in \N$ and $q$ and positive constants $C_1\le C_2$ if there exist a set 
$$
\Big\{\xi^j \in \Omega: j=1,\dots,m\Big\}
$$ 
 such that for any $f\in X_N$ we have
\be\label{D3}
C_1\|f\|_q^q \le \frac{1}{m} \sum_{j=1}^m |f(\xi^j)|^q \le C_2\|f\|_q^q.
\ee
In the case $q=\infty$ we define $L_\infty$ as the space of continuous functions on $\Omega$  and ask for
\be\label{D4}
C_1\|f\|_\infty \le \max_{1\le j\le m} |f(\xi^j)| \le  \|f\|_\infty.
\ee
We will also use the following brief way to express the above properties: The $\cM(m,q)$ (more precisely the $\cM(m,q,C_1,C_2)$) theorem holds for  a subspace $X_N$, written $X_N \in \cM(m,q)$ (more precisely $X_N \in \cM(m,q,C_1,C_2)$).

In \cite{VT180} bound (\ref{D2}) was derived from the following general result (see Lemma 3.2 there).
Let $X_N$ be an $N$-dimensional subspace of $\C(\Omega)$. Denote by $X_N^q$ the unit $L_q$-ball of the $X_N$. 

\begin{Lemma}\label{DL2} Let $q\in (2,\infty)$. Assume that for any $f\in X_N$ we have
\be\label{D5}
\|f\|_\infty \le M\|f\|_q
\ee
with some constant $M$. Also, assume that $X_N \in \cM(s,\infty,C_1)$ with $s\le aN^c$.
Then for $k\in [1,N]$ we have  
\be\label{D6}
\e_k(X_N^q, L_\infty) \le C(q,a,c,C_1)  M \left(\frac{\log N}{k}\right)^{1/q} .
\ee
\end{Lemma}

Note that Lemma \ref{DL2} is based on deep results from functional analysis (see \cite{VT180}, Lemma 3.1; \cite{Kos}, Corollary 4.2;  
  \cite{Tal}, p.552, Lemma 16.5.4). 
Lemma \ref{DL1} follows from Lemma \ref{DL2} and
 known Nikol'skii inequality for $\Tr(Q_n)$ (see \cite{Tmon}): For any $f\in \Tr(Q_n)$ we have 
$$
\|f\|_\infty \ll  2^{n/q} n^{(d-1)(1-1/q)}\|f\|_q.
$$
In this paper we gave other proof of Lemma \ref{DL1}, which is simpler and more elementary than the mentioned above known proof. 
We stress that for applications in sampling discretization of integral norms it is important that Lemma \ref{DL1} has form (\ref{D2}). Indeed, the following general conditional result is used for such applications.  The following Theorem \ref{DT1} in case $q=1$ was proved in  \cite{VT159} and it was extended to the case 
$q\in (1,\infty)$ in \cite{DPSTT1}. 

\begin{Theorem}\label{DT1} Let $1\le q<\infty$. Suppose that a subspace $X_N$ satisfies the condition
\be\label{D7}
\e_k(X^q_N,L_\infty) \le  B (N/k)^{1/q}, \quad 1\leq k\le N,
\ee
where $B\ge 1$.
Then for large enough constant $C(q)$ there exists a set of
$$
m \le C(q)NB^{q}(\log_2(2BN))^2
$$
 points $\xi^j\in \Omega$, $j=1,\dots,m$,   such that for any $f\in X_N$
we have
$$
\frac{1}{2}\|f\|_q^q \le \frac{1}{m}\sum_{j=1}^m |f(\xi^j)|^q \le \frac{3}{2}\|f\|_q^q.
$$
\end{Theorem}

The reader can find results on sampling discretization of integral norms of trigonometric polynomials with frequencies from a hyperbolic cross in a recent paper \cite{VT180}. For general results on the entropy numbers of the unit $L_p$-balls, $1\le p\le 2$, of finite dimensional subspaces in the uniform norm we refer the reader to \cite{DPSTT2} (see Theorem 2.1 there). Some results on the entropy numbers of the unit $L_q$-balls of $\Tr(Q_n)$ in the norm $L_p$, $1<q<p<\infty$, can be found in 
\cite{VTbookMA}, Ch.7. 

{\bf Open problem 3.} Could we improve Theorem \ref{DT1} if instead of condition (\ref{D7}) imposed on the entropy numbers $\e_k(X^q_N,L_\infty)$ we use the same condition imposed on the Kolmogorov widths $d_k(X^q_N,L_\infty)$?

{\bf Acknowledgements.}  
The first author was supported by the Russian Federation Government Grant N{\textsuperscript{\underline{o}}}14.W03.31.0031. The second author was supported by the DFG Ul-403/2-1 grant. The paper contains results obtained in frames of the program \lq\lq Center for the storage and analysis of big data", supported by the Ministry of Science and High Education of Russian Federation (contract 11.12.2018 N{\textsuperscript{\underline{o}}}13/1251/2018 between the Lomonosov Moscow State University and the Fund of support of the National technological initiative projects).


\begin{thebibliography}{9999}
 
 \bibitem{C} B. Carl, Entropy numbers, $s$-numbers, and eigenvalue problems, J. Func. Analysis, {\bf 41} (1981), 290--306.
 
 \bibitem{CM} A. Cohen and G. Migliorati, Optimal weighted least-squares methods, {\it SMAI J. Computational Mathematics} {\bf 3} (2017), 181--203.
 
  \bibitem{DPTT} F. Dai, A. Prymak, V.N. Temlyakov, and  S.U. Tikhonov, Integral norm discretization and related problems,
  {\it Russian Math. Surveys} {\bf 74:4} (2019),   579--630.
 Translation from
{\it Uspekhi Mat. Nauk}  {\bf 74:4(448)}  (2019),	3--58; arXiv:1807.01353v1.

\bibitem{DPSTT1} F. Dai, A. Prymak, A. Shadrin, V. Temlyakov, S. Tikhonov,
Sampling discretization of integral norms,
arXiv:2001.09320v1 [math.CA] 25 Jan 2020.

 \bibitem {DPSTT2} F. Dai, A. Prymak, A. Shadrin, V. Temlyakov, and S. Tikhonov, Entropy numbers and Marcinkiewicz-type discretization theorem, arXiv:2001.10636v1 [math.CA] 28 Jan 2020.
 
   \bibitem{DU} Dinh D{\~u}ng and T. Ullrich, Lower bounds for the integration error for multivariate functions with mixed smoothness and optimal Fibonacci cubature for functions on the square,
Math. Nachr., {\bf 288} (2014), 743--762. 

 \bibitem{DTU} Dinh D{\~u}ng, V.N. Temlyakov, and T. Ullrich, Hyperbolic Cross Approximation, Advanced Courses in Mathematics CRM Barcelona, Birkh{\"a}user, 2018; arXiv:1601.03978v2 [math.NA] 2 Dec 2016.
 
 \bibitem{FPRU} S. Foucart, A. Pajor, H. Rauhut, and T. Ullrich, The Gelfand widths of $\ell_p$-balls for $0<p\le 1$, J. Complexity, {\bf 26}, 629--640.
 
  \bibitem{Kos} E. Kosov, Marcinkiewicz-type discretization
of $L^p$-norms under the Nikolskii-type inequality assumption, arXiv:2005.01674v1 [math.FA] 4 May 2020.

\bibitem{KU} D. Krieg and M. Ullrich, Function values are enough for $L_2$-approximation,
arXiv:1905.02516v4 [math.NA] 19 Mar 2020. 

\bibitem{KU2} D. Krieg and M. Ullrich, Function values are enough for $L_2$-approximation: Part II,
arXiv: 2011.01779v1 [math.NA] 3 Nov 2020.

\bibitem{LT} I. Limonova and V. Temlyakov, On sampling discretization in $L_2$, arXiv:2009.10789v1 [math.FA] 22 Sep 2020. 

\bibitem{No} E. Novak, Quadrature and Widths, J. Approx. Theory, {\bf 47} (1986), 195--202.

\bibitem{NoLN} E. Novak, {\em Deterministic and Stochastic Error Bounds in
Numerical Analysis}, Springer-Verlag, Berlin, 1988.

\bibitem{NSU} N. Nagel, M. Sch{\"a}fer, T. Ullrich, A new upper bound for sampling numbers,
arXiv:2010.00327v1 [math.NA] 30 Sep 2020. 

 \bibitem{Tal} M. Talagrand, Upper and lower bounds for stochastic processes: modern methods and classical problems.
-- Springer Science and Business Media, 2014.

\bibitem{Tmon} V.N. Temlyakov, Approximation of functions with bounded mixed derivative, Trudy MIAN, {\bf 178} (1986), 1--112. English transl. in Proc. Steklov Inst. Math., {\bf 1} (1989).

 \bibitem{VT43}  V.N. Temlyakov,  On a way of obtaining lower estimates 
for the errors of quadrature formulas, Matem. Sbornik, {\bf 181} (1990),
 1403--1413;  English transl. in  Math.
USSR Sbornik, {\bf 71} 
(1992). 

\bibitem{VTbook} V.N. Temlyakov, Greedy Approximation, Cambridge University
Press, 2011

 \bibitem{VT150} V.N. Temlyakov, Constructive sparse trigonometric approximation and other problems for functions with mixed smoothness,  Matem. Sb., {\bf 206} (2015), 131--160;  arXiv: 1412.8647v1 [math.NA] 24 Dec 2014, 1--37. 
 
\bibitem{VT159} V.N. Temlyakov, The Marcinkiewicz-type discretization theorems, {\it Constructive Approximation}, {\bf 48} (2018), 337--369. 

\bibitem{VTbookMA} V.N. Temlyakov, Multivariate Approximation, Cambridge University Press, 2018.

\bibitem{VT180} V.N. Temlyakov, Sampling discretization of integral norms of the hyperbolic cross polynomials, arXiv:2005.05967v1 [math.NA] 12 May 2020.

\bibitem{VT183} V.N. Temlyakov, On optimal recovery in $L_2$, arXiv:2010.03103v1 [math.NA] 7 Oct 2020. 

\bibitem{TrBe} R.M. Trigub and E.S. Belinsky, Fourier Analysis and Approximation of Functions, Kluwer Academic Publishers, 2004. 
 
 	
 \end{thebibliography}
\end{document}